\documentclass[11pt]{article}
\usepackage{graphicx}
\usepackage{caption}
\usepackage{amsfonts,amsmath}
\usepackage{amssymb}
\usepackage{fancyhdr}
\usepackage{titlesec}
\usepackage{indentfirst}
\usepackage{booktabs}
\usepackage{verbatim}
\usepackage{color}
\usepackage{tcolorbox}
\usepackage{mathtools}
\usepackage{bm}
\usepackage{amsthm}
\usepackage{wasysym}
\usepackage{empheq}
\usepackage{bbold}
\graphicspath{{figures/}} 
\usepackage[font=small,labelfont=bf]{caption} 
\usepackage{chngcntr}

\newcommand{\rom}[1]{\uppercase\expandafter{\romannumeral #1\relax}}
\usepackage[page,header]{appendix}
\usepackage{subfig}
\usepackage{titletoc}
\usepackage{bbold}

\usepackage{indentfirst,enumerate}
\usepackage{mathrsfs, subeqnarray,setspace,epstopdf,
color,float,algorithmicx,hyperref}
\usepackage{algpseudocode}
\newtheorem{theorem}{Theorem}[section]
\newtheorem{lemma}{Lemma}[section]

\newtheorem{proposition}{Proposition}[section]

\newtheorem{algorithm}{Algorithm}[section]
\newtheorem{remark}{Remark}[section]
\newtheorem{assumption}{Assumption}[section]

\newtheorem{step}{Step}
\numberwithin{equation}{section}
\numberwithin{theorem}{section}

\newcommand{\argmin}{\operatornamewithlimits{argmin}}
\newcommand{\argmax}{\operatornamewithlimits{argmax}}

\usepackage{tikz}

\def\b{\boldsymbol}

\def\e{\epsilon}

\def\R{\mathbb{R}}

\def\E{\mathbb{E}}

\def\a{\beta}

\def\b{\beta}

\def\e{\epsilon}
\def\g{\gamma}
\def\lam{\lambda}

\def\s{\sigma}

\def\E{\mathbb E}

\def\R{\mathbb R}

\def\l{\left}
\def\r{\right}

\def\lV{\left\lVert}
\def\rV{\right\rVert}
\def\lv{\left\lvert}
\def\rv{\right\rvert}
\def\({\left(}
\def\){\right)}

\def\nb{\nabla}

\def\qd{\quad}

\def\h{\hat}

\def\hL{\hat{L}}

\def\hLj{\hat{L}^j}

\def\w{x}
\def\bws{\bar{x}^*}
\def\ws{x^*}
\def\W{X}
\def\Wj{X^j}

\def\tt{{\text{test}}}

\begin{document}

\title{A consensus-based global optimization method for high dimensional machine learning problems}

\author{Jos\'e A. Carrillo\thanks{Department of Mathematics, Imperial College London,
London SW7 2AZ, United Kingdom. Email: carrillo@imperial.ac.uk} 
\qd 
Shi Jin\thanks{School of Mathematical Sciences, Institute of Natural Sciences, MOE-LSC, Shanghai Jiao Tong University, Shanghai, 200240, P. R. China. Email: shijin-m@sjtu.edu.cn}
\qd 
Lei Li\thanks{School of Mathematical Sciences, Institute of Natural Sciences, MOE-LSC, Shanghai Jiao Tong University, Shanghai, 200240, P. R. China. Email: leili2010@sjtu.edu.cn}
\qd 
Yuhua Zhu\thanks{Department of Mathematics, Stanford University, California, 94305, United States. Email: yuhuazhu@stanford.edu}
}

\date{}
\maketitle

\begin{abstract}
We improve recently introduced consensus-based optimization method, proposed in [R. Pinnau, C. Totzeck, O. Tse and S. Martin, Math. Models  Methods Appl. Sci., 27(01):183--204, 2017], which is a gradient-free optimization method for general non-convex functions. We first replace the isotropic geometric Brownian motion by the component-wise one, thus removing the dimensionality dependence of the drift rate, making the method more competitive for high dimensional optimization problems. Secondly, we utilize the random mini-batch ideas to reduce the computational cost of calculating the weighted average which the individual particles tend to relax toward.  
For its mean-field limit--a nonlinear Fokker-Planck equation--we prove, in both time continuous and semi-discrete settings,  that the convergence of the method, which is exponential in time, is guaranteed with parameter constraints {\it independent} of the dimensionality. We also conduct numerical tests to high dimensional problems to check the success rate of the method.
\end{abstract}

\section{Introduction}

Our main goal in this work is developing gradient-free optimization methods to the following classical unconstrained optimization problem
\begin{equation}
\label{eq: main opt}
    \ws = \argmin_{\w \in \R^d} L(\w)\,,
\end{equation}
in {\it high dimesions}, where the target function, not necessarily convex,  is assumed to be a continuous function defined on $\R^d$ achieving a unique global minimum. Target functions defined on subsets of $\R^d$ can be extended to the whole space recasting the corresponding optimization problem in the form \eqref{eq: main opt}. Moreover, we can assume without loss of generality that the target function is positive, i.e., $L(\ws)>0$, by lifting $L$ by a suitable constant. 

Important examples of target functions stem from machine learning and artificial intelligence applications. Typical neural training networks lead to optimization for functions of the following form
\begin{equation}\label{eq: ml obj}
     L(\w) = \frac{1}{n}\sum_{i= 1}^n \ell(f(\w, \h{x}_i), \h{y}_i)=: \frac{1}{n}\sum_{i= 1}^n \ell_i(\w),
\end{equation}
where $\w$ is the set of parameters defining the model, $(\h{x}_i, \h{y}_i)_{i=1}^n$ constitute the training data set, the function $f(\w,\h{x})$ defines the neural network that one wants to learn, and the function $\ell(f,y)$ is the loss function measuring the distance between the prediction $f(\w, \h{x}_i)$ and the observations $\h{y}_i$.

For such optimization problems, gradient-based methods have been dominating. Nevertheless, in general, most gradient-based methods have problems dealing with functions that have large noise or non-differentiable functions. They are also not designed to handle multi-modal problems or discrete and mixed discrete-continuous design variables \cite{DO}. More specifically in machine learning problems, it has been proved that as the deep neural network gets deeper, the gradient tends to explode or vanish \cite{bengio1994learning,hanin2018neural}. Besides, it will be easily influenced by the geometry of the landscape \cite{liu2019bad}. 

On the other hand, there are also gradient-free methods such as Nelder-Mead (NM) method \cite{nelder1965simplex}, genetic algorithm (GA) \cite{holland1992genetic}, simulated annealing (SA) \cite{kirkpatrick1983optimization}, particle swarm optimization (PSO)\cite{eberhart1995particle, kennedy2006swarm}, etc..  The NM method is a direct search method based on function comparison; GA is inspired by genetic evolution and are commonly used to generate high-quality solutions to optimization; SA is a probabilistic technique for approximating the global optimum, which is often used in discrete search space; PSO is used to model the flocking behavior of birds, and also found to be a good optimization method. 

One of such gradient-free methods is the concensus-based optimization (CBO) method, established in \cite{pinnau2017consensus,carrillo2018analytical,TPBS}. This is a method based on an interacting particle system, along the line of consensus based models \cite{bellomo2013,MR3304346,MR2596552,MR2744704,MR2295620,MR2425606,MR3143990, MR3274797,MR2247927}. This particle system consists of $N$-particles, labeled as $X^j, j=1, \cdots N$, that tend to relax  toward their weighted average, and
in the meantime also undergo fluctuation with a multiplicative noise:
\begin{gather}\label{eq:cctt}
dX^j=-\lambda (X^j-\bar{x}^*) H^{\e}(L(X^j)-L(\bar{x}^*))\,dt +\sigma|X^j-\bar{x}^*| dW^j,
\end{gather}
where $\bar{x}^*$ is the weighted average of the positions of the particles according to 
\begin{gather}\label{eq:weightedmean}
\bar{x}^*=\frac{1}{\sum_{j=1}^N e^{-\beta L(X^j)}}\sum_{j=1}^N X^j e^{-\beta L(X^j)}.
\end{gather}
The function $H^{\e}$ is a regularization of the Heaviside function introduced by the authors with the objective that the particles will drift only if at their positions the cost value is higher than the average of all particles. Later the authors in \cite{carrillo2018analytical} considered this model without the Heaviside cutoff for the convenience of analysis. The diffusion given in \eqref{eq:cctt} is associated with $|X^j-\bar{x}^*|$ and this yields convergence conditions and methods depending on the dimension $d$, see \cite[Theorem 4.1]{carrillo2018analytical}. The motivation for this type of diffusion is that one wants to explore the landscape of the cost function $L(x)$ if one is far from consensus, but when consensus forms one wants the noise to decrease, and eventually to disappear, in order to stabilize the results towards the target $\ws$. This decrease of the temperature is a common feature with simulated annealing \cite{MR904050,MR995752,hwang1990large}.

Here, $e^{-\beta L(x)}$ is the Gibbs distribution corresponding to $L(x)$. The motivation for this choice comes from statistical mechanics: the cost function $L(x)$ corresponds to a potential in which particles move by steepest descent modulated by Brownian noise with $\beta$ being the inverse of the temperature leading to this invariant measure. In this way, the smaller the value of the temperature is, the larger the weight of the normalized Gibbs measure for the agents is on the minimum value of the cost function $L(x)$. The quantitative formulation of this intuition is given by the Laplace principle \cite{BO,miller,dembo2009large}, a classical asymptotic method for integrals, recalled here for reader's sake: for any probability measure $\rho\in\mathcal{P}(\R^d)$ compactly supported with $x_*\in\text{supp}(\rho)$, then
\begin{gather}\label{eq:laplaceprinciple}
 \lim_{\beta\to\infty} \left( -\frac{1}{\beta}\log\left(\int_{\R^d} e^{-\beta L(x)} d\rho(x) \right)\right) = L(\ws) > 0.
\end{gather}
Therefore, if $L$ attains its minimum at a single point $\ws\in\text{supp}(\rho)$, then the suitably normalized measure $e^{-\beta L(x)} \rho$ assigns most of its mass to a small region around $\ws$  and hence we expect it  approximates a Dirac distribution $\delta_{{\bws}}$ for large $\beta\gg1$. Consequently, the first moment of the normalized measure $e^{-\beta L(x)} \rho$, and thus, the discrete counterpart average $\bar{x}^*$, should provide a good estimate of the point at which the global minimum is attained,  $\ws =\argmin L$. Furthermore, the convergence rate toward the global minimum is {\it exponential} in time. However, to guarantee the
convergence of this method, {\it the drift rate $\lambda$ depends on the dimension parameter $d$}, which makes the particles move away from its global equilibrium more easily  for high dimension problems in which $d\gg 1$, such as those raising from
machine or deep learning problems. 

In this paper, we improve the above CBO algorithms in two ways. First, we replace the isotropic Brownian motion term $\sigma|X^j-\bar{x}^*| dW^j$, which is added equally in all dimensions, by its component-wise counterpart. For its mean-field limit equation, in both time continuous and a time semi-discrete settings, we prove that this removes the $d$-dependence constraint on $\lambda$. Secondly, we utilize the random mini-batch ideas, an essential ingredient in stochastic gradient descent (SGD) method \cite{RM51,bottou1998online,bubeck2015convex}, and also introduced recently in \cite{jin2018random} for interacting particle systems, that reduces the computational cost in calculating $\bar{x}^*$ from $O(N)$ to $O(1)$, or $O(nN)$ to $O(1)$ in the case of \eqref{eq: ml obj}, and thus it reduces the overall computation cost of CBO. The noise introduced by the random selection of the mini-batches also adds extra stochasticity which makes the particles more likely to escape the local equilibrium, thus enhances the success rate of the algorithm toward the global minimum, even by one order of magnitude in some cases as shown in subsection 4.2. 

We point out that more recently, in \cite{HJK19}  the consensus of particles, and convergence toward the global minimum under dimension-independent
conditions on parameters and initial data, are established for the particle system (2.2), without using the mean-field limit.

The paper is organized as follows. We present our algorithm and its continuous model in Section \ref{sec: model algo}. We prove in Section \ref{sec: analysis} that in the continuous and a semi-discrete in time settings {the mean-field limit of the algorithm} will converge to the global minimum exponentially fast, with a constraint on $\lambda$ that is independent of $d$.  We give several numerical experiments in Section \ref{numerics} to verify the performance and efficiency of the algorithm. 


\section{A new Consensus Based Optimization Method}
\label{sec: model algo}

Our first new observation over the CBO model (\ref{eq:cctt}) is that if one uses the {\it component-wise geometric Brownian motion}, which we shall clarify soon, the dimension dependence in the convergence estimates \cite[Theorem 4.1]{carrillo2018analytical} can be dramatically reduced. To illustrate these ideas,  let us fix $\bar x^*=a$ to be a constant vector and consider solely the effect of the diffusion term. Let us consider a shifted second moment for \eqref{eq:cctt} in the case of $H\equiv 1$ to obtain
\[
\frac{d}{dt}\E|X-a|^2=-2\lambda \E|X-a|^2+\sigma^2 \sum_{i=1}^d \E |X-a|^2
=(-2\lambda+\sigma^2 d)\E|X-a|^2.
\]
Clearly, if the particles are to form consensus, one needs $2\lambda>d\sigma^2$. 
Now consider the SDE with component-wise geometric Brownian motion
\begin{gather}\label{eq:sdecomponentgbm}
dX=-\lambda(X-a)\,dt+\sigma \sum_{k=1}^d  (X-a)_k dW_k \vec{e}_k,
\end{gather}
where $(X-a)_k$ means the $k$th component of $X-a$, $\{W_k\}_{k=1}^d$ are independent standard Brownian motions, and $\vec{e}_k$ is the unit vector along the $k$th dimension. 
For the interacting particle system \eqref{eq:sdecomponentgbm}, one easily finds that
\[
\frac{d}{dt}\E|X-a|^2=-2\lambda \E|X-a|^2+\sigma^2 \sum_{i=1}^d \E (X-a)_i^2
=(-2\lambda+\sigma^2)\E|X-a|^2.
\]
We only need $2\lambda>\sigma^2$ for the particles to concentrate. 
The restriction between $\lambda$ and $\sigma$ is {\it dimension $d$ insensitive} for the particles to concentrate (or form a consensus as the terminology in \cite{pinnau2017consensus}).  

Based on this observation, we now propose a modification to the CBO model in \eqref{eq:cctt} together with an efficient algorithm for its computation by the random batch approach championed in \cite{jin2018random}. We tweak the CBO method introduced in \cite{pinnau2017consensus,carrillo2018analytical} by considering the following model with diffusion corresponding to a {\it component-wise} geometric Brownian motion
\begin{equation}\label{eq:optmodel}
\begin{aligned}
&dX^j=-\lambda (X^j-\bar{x}^*)\,dt+\sigma \sum_{k=1}^d  (X^j-\bar{x}^*)_k dW^j_k \vec{e}_k \,,\\
&{\bar{x}^*=\frac{1}{\sum_{j=1}^N e^{-\beta L(X^j)}}\sum_{j=1}^N X^j e^{-\beta L(X^j)}.}
\end{aligned}
\end{equation}
Here {$\bar{x}^*$ is the same as \eqref{eq:weightedmean}}. 

Compared to \eqref{eq:cctt}, this new model has a simpler and cleaner form, where we omit the Heaviside function as in \cite{carrillo2018analytical} and use the component-wise geometric Brownian motion to replace their noise.  From the viewpoint of opinion models in social sciences, the interacting particle system   \eqref{eq:optmodel} may be thought as the society drifting according to some common sense opinion or command stemming from the individual parties. Let us remark that $\bar{x}^*$ can be chosen to be updated only at some discrete time points in practice without changing the spirit of the modeling.  

As already commented, the usage of the diffusion according to {\it component-wise} geometric Brownian motion is largely due to its scalability in high dimension space. In fact, the assumptions and the results later in Theorem \ref{thm:timecontinuous} for the particles to converge are all {\it independent of the dimensionality $d$} of the particle $X$. The model \eqref{eq:cctt} on one hand requires $\sigma^2$ to be small or large $\lambda$ which reduces the exploration ability at the initial stage. Moreover, it will also put more severe constraint on the parameters, especially on the second moments of the initial condition of $X_0$. Besides, the use of the diffusivity in \eqref{eq:optmodel} allows the particles to explore each dimension with different rate, and possibly easier to find the optimizer.  To summarize, we expect the optimization method \eqref{eq:optmodel} to be efficient for optimization problems where the dimensionality of the parameter is very high, such as those in deep learning compare to (\ref{eq:cctt}). 

The computational cost of a straightforward numerical scheme to approximate the new continuous CBO method \eqref{eq:optmodel} is too high if we do \eqref{eq:weightedmean} in every time step, especially for big data problems in the form of (\ref{eq: ml obj}). Hence, our second novel approach is to apply the random batch method \cite{jin2018random} to the interacting particle system \eqref{eq:optmodel}, which leads to efficient random algorithms. See also \cite{RM51,bottou1998online,bubeck2015convex} for the relevant mini-batch ideas used for SGD. These random algorithms can also be viewed as new models, which seem to be closer to opinion models in social sciences. 

The random mini-batch strategy developed in \cite{jin2018random} will be extended to two levels for the typical cost functions arising in machine learning such as \eqref{eq: ml obj}. First, we calculate the empirical expectation $\hLj = \hL(\Wj)$ from a random subset of the training data set instead of the accurate $L^j$ as the objective value for the ensemble of particles $\Wj$; second, we apply the mini-batch approach and update the reference $\bws$ by a random subset of the particle ensemble instead of all particles. These two modifications allow us to do high dimensional optimization more efficiently. The mini-batch is done without replacement, that is, we do a random permutation and then select the mini-batch in order. Let us finally remark that this random choice of subsets of interacting particles is very similar to Monte Carlo approaches to compute large averages, and and it has been used to produce efficient algorithms for mean-field (kinetic) swarming models in \cite{AP,CPZ}.

We introduce below the random algorithm to solve in practice our new CBO model \eqref{eq:optmodel} and \eqref{eq:weightedmean}.
\begin{algorithm}
\label{algo: main algo} 
Generate $\{\Wj_0 \in \R^d\}_{j=1}^N$ according to the same distribution $\rho_0$.   
Set the remainder set $\mathcal{R}_0$ to be empty. For $k=0,1,2, \cdots $, do the following:

\step Concatenate $\mathcal{R}_k$ and a random permutation $\mathcal{P}_k$ of the indices $\{1, 2,\cdots, N\}$ to form a list $\mathcal{I}_k=[\mathcal{R}_k, \mathcal{P}_k]$.  Pick $q=\lfloor \frac{N+|\mathcal{R}_k|}{M}\rfloor$ sets of size $M\ll N$ from the list $\mathcal{I}_k$ in order to get batches $B_1^k, B_2^k, \cdots, B_q^k$ and set the remaining indices to be $\mathcal{R}_{k+1}$. Here, $|\mathcal{R}_k|$ means the number of elements in $\mathcal{R}_k$.

\step For each $B_{\theta}^k$ ($\theta=1,\cdots, q$), do the following

   \begin{enumerate}
   \item Calculate the function values (or approximated values) of $L$ at the location of the particles in $B_\theta^k$ by $L^j:= L(\Wj),~\forall j\in B_\theta^k$.
     If $L(\w)$ is in the form \eqref{eq: ml obj} with $n\gg 1$, one then applies the random mini-batch idea again: 
generate a random index subset $A_\theta^k \subset \{1,\cdots, n\}$ with $|A_\theta^k| = m$, and approximate $L^j$ for all $j\in B_\theta^k$ by
\begin{equation*}
	\hLj := \hL_\theta^k(\Wj) = \frac{1}{m}\sum_{i\in A_\theta^k} \ell_i(\Wj),~\forall j\in B_\theta^k\,,
\end{equation*}
where $\hL^k_\theta(\w):=\frac{1}{m}\sum_{i\in A_\theta^k} \ell_i(\w)$ is an unbiased approximation to $L(\w)=\frac{1}{n}\sum_{i=1}^n \ell_i(\w)$ defined in \eqref{eq: ml obj}.

  \item  Update $\bws_{k,\theta}$ according to the following weighted average, 
\begin{equation}
        \label{eq: ws exponential}
        \bws_{k,\theta} = \frac{1}{\sum_{j\in B_{\theta}^k} \mu_j}\sum_{j\in B_{\theta}^k} \Wj \mu_j ,\qd \text{with}\qd  \mu_j = e^{-\beta L^j} \text{~or~} e^{-\beta \hLj}. 
\end{equation}
     \item Update $\Wj$ for $j\in \mathcal{J}_{k,\theta}$ as follows, 
     \begin{equation}\label{discrete swarming}
             \Wj \leftarrow \Wj - \lam \g_{k,\theta}( \Wj - \bws_{k,\theta})+ \sigma_{k,\theta} \sqrt{\g_{k,\theta}} \sum_{i=1}^d \vec{e}_i\(\Wj - \bws_{k,\theta}\)_i    z_i^j,~~z_i^j\sim\mathcal{N}(0, 1),
      \end{equation}
      where $\g_{k,\theta}$ is the learning rate chosen suitably and there are two options for $\mathcal{J}_{k,\theta}$:
       \begin{equation*}
              \begin{aligned}
                   &\text{{\it partial updates: }}\qd \mathcal{J}_{k,\theta}=B_{\theta}^k,\\
                   &\text{{\it full updates: }}\qd  \mathcal{J}_{k,\theta}=\{1,\cdots, N\}.
              \end{aligned}    
        \end{equation*}

   \end{enumerate}

   \step Check the {\bf Stopping criterion:} 
\begin{equation*}
 \frac{1}{d} \lV \Delta\bws \rV_2^2 \leq \e,
\end{equation*}
where $\lV \cdot \rV_2$ is the Euclidean norm and $\Delta\ws$ is the difference between two most recent $\bws_{k,\theta}$.
If this is not satisfied, repeat Steps 1-2.

\end{algorithm}

Note again that $\(\Wj - \bws_{k,\theta}\)_i$ represents for the $i$th component of the vector in \eqref{discrete swarming}. $\lam$ is the drift rate, $\g_{k,\theta}$ is the learning rate, $\s_{k,\theta}$ is the noise rate and $z_i$ is a random variable following the standard normal distribution. Note that we add $\sqrt{\gamma_{k,\theta}}$ on purpose to be consistent with the time-continuous model \eqref{eq:optmodel}. These parameters can be different from step to step in practice, as often used in machine learning and optimization. {In practice, one often chooses $\gamma_{k,\theta}$ in a decreasing fashion satisfying $\sum_k \gamma_{k,\theta}=\infty$. In our experiments, we fix $\gamma_{k,\theta}\equiv \gamma$ to be constant. In general, it needs to be chosen to satisfy a numerical stability condition. One stability result is given in \cite{HJK19}.} Besides, decreasing $\sigma_{k,\theta}$ slowly corresponds to the famous simulated annealing algorithm in optimization \cite{MR904050,MR995752,hwang1990large}.\\

\begin{remark}
The estimated value $\hL$ of the objective function is especially efficient for problems of the form \eqref{eq: ml obj}. Usually, to train a good model, one requires a large number of data, that is, $n\gg 1$. The computational cost would be high if one calculate $L(\w)$ at each step for all particles. If we calculate $\hL$ based on a small subset of the data, the computational cost will be largely saved. Besides, we will show later in the numerical experiments that using $\hL$ can not only save computational cost, but also make the algorithm converge to the optimizer faster, due to stochasticity introduced by randomly selecting the mini-batches. This is an established concept for algorithms such as the SGD \cite{bottou1998online,bubeck2015convex}.
\end{remark}

\begin{remark}
An alternative way to update $\bws$ is to let it equal to $\argmin \hL_j$, that is, 
\begin{equation}
        \label{eq: ws argmin}
        \bws_k = \argmin_{\Wj\in B_\theta^k}\hL(\Wj).
\end{equation}
We will show that it numerically performs as good as the penalized average.  We will leave the theoretical proof of this case for future study.
\end{remark}

\begin{remark}
\label{rmk: step3}
For updating $X^j$:
\begin{itemize}
\item There are two ways to introduce extra noises into the algorithm. One way is to let particles do geometric Brownian motion as in (\ref{discrete swarming}), another way is let particles do a Brownian motion only when $\Wj$ stops moving forward. The reason why the second method also works is because we already introduced noise by using the estimated $\hLj$ and (\ref{eq: ws exponential}) using the randomly generated sets $B_k$, so the noise term in (\ref{discrete swarming}) is sometimes not necessary.  We will show later in the numerical experiments that if we do not have the last term in  (\ref{discrete swarming}) and just add a Brownian motion when $\Wj$ stops moving forward, the performance is still good. 

\item In some optimization problem, since the landscape of the objective function is too complicated (for example, the MNIST data in Section \ref{sec: MNIST}), it cannot converge to the global minimizer at stoppting time. Therefore, when $\bws$ stops updating, we record $\hL(\bws)$ at that step, and make all particles do an independent Brownian motion, i.e. for $\forall j$,
\begin{equation*}
\Wj \leftarrow \Wj + \sigma_{k,\theta} \sqrt{\g} \sum_{i=1}^d \vec{e}_i z_i^j,~~~z_i^j\sim\mathcal{N}(0, 1),
\end{equation*}
then repeat the algorithm. We terminate the procedure if the recorded $\hL(\bws)$ is not decreasing any more. 

\end{itemize}
\end{remark}

\section{Analysis of the mean-field limit models}
\label{sec: analysis}
The analysis of the computational model in Section \ref{sec: model algo} is quite challenging: analyzing the $N$-particle system, showing the existence of singular invariant measures quantifying the convergence towards them would be a fantastic breakthrough. In this section, we will consider the formal mean-field limit models ($N\to\infty$) of the interacting particle system \eqref{eq:optmodel}, {which makes the analysis possible even if working with high dimensional PDEs, as already shown in \cite{carrillo2018analytical}.} We remark that the rigorous proof of the mean-field limit is another open problem for these interacting particle systems due to the difficulty of managing the multiplicative noise term in \eqref{eq:optmodel}. Depending on how we treat the time variable, one can write a time-continuous model and a semi-continuous mean-field models, as discussed below.

\subsection{Time continuous model}

Formally, taking $N\to\infty$ in the model \eqref{eq:optmodel} with full batch (or alternatively, $\gamma\to 0$ and $N\to\infty$ in Algorithm \ref{algo: main algo} with full batch), the mean field limit of the model is formally given by the following stochastic differential equation for $X=X(t)$:
\begin{gather}\label{eq: cts swarming}
dX =-\lambda(X - \bws)dt + \sigma \sum_{i=1}^d \vec{e}_i (X - \bws)_i dW_i , 
\end{gather}
where
\begin{gather}\label{eq: cts swarming2}
\bws=\frac{\E(Xe^{-\beta L(X)})}{\E(e^{-\beta L(X)})}.
\end{gather}
We refer to \cite{MR2744704} and {\cite{BCC,MR3331178,HJ1,J,JW} for formal and rigorous discussions respectively of the results about mean-field models.}
The law $\rho(\cdot,t)$ of the process $X(t)$ in the high dimensional space $\R^d$ solving the nonlinear stochastic differential equation \eqref{eq: cts swarming}-\eqref{eq: cts swarming2} follows the nonlinear Fokker-Planck equation
$$
\partial_t\rho=\lambda \nabla\cdot((x-\bws)\rho)+\frac{1}{2}\sigma^2\sum_{i=1}^d\partial_{ii}((x-\bws)_i^2 \rho),
$$
where
\[
\bws=\frac{\int_{\R^d} xe^{-\beta L(x)}\rho(x,t)\,dx}{\int_{\R^d} e^{-\beta L(x)}\rho(x,t)\,dx}.
\] 
We will prove that the stochastic process $X(t)$ will approach some point $\tilde{x}$, which is an approximation of $\argmin L(x)$. Our proof needs the following assumptions:
\begin{assumption}\label{ass:modelcond}
We assume that $L_m:=\inf L>0$, without loss of generality,  and that the cost function $L$ satisfies
$c_L:=\max(\| \max_i|\partial_{ii}L|\|_{\infty}, \| r(\nabla^2 L)\|_{\infty})<\infty$.
\end{assumption}

Here, $\nb^2L$ represents the Hessian of $L$, and $r(\nabla^2 L)$ is the spectral radius while $\partial_{ii}L$ is the diagonal element of $\nb^2L$. Our assumption means that these two quantities should be of the same order. One sufficient condition is that all second derivatives of $L$ are bounded. 
Let us also define the following averaged quantities:
\begin{gather}\label{eq:varianceandML}
V(t):=\E|X-\E X|^2 \qquad \mbox{and} \qquad M_L(t):=\E e^{-\beta L(X)}.
\end{gather}
Here, $V$ is the variance of the process $X$, while $M_L(t)$ is the total weight in the optimization. The following result gives the convergence of the continuous mean field model \eqref{eq: cts swarming}, whose proof is deferred to Appendix \ref{app:meanfield1} following the blueprint in \cite{carrillo2018analytical}. 

\begin{theorem}\label{thm:timecontinuous}
If $\beta,\lambda, \sigma$ and the initial distribution are chosen such that
\begin{gather}\label{eq:ass}
\begin{split}
& \mu:=2\lambda-\sigma^2-2\sigma^2\frac{e^{-\beta L_m}}{M_L(0)}>0,\\
& \nu:=\frac{2V(0)}{\mu M_L^2(0)}\beta e^{-2\beta L_m} c_L(2\lambda+\sigma^2)\le \frac{3}{4},
\end{split}
\end{gather}
then $V(t)\to 0$ exponentially fast and there exists $\tilde{x}$ such that
$
\bws(t)\to \tilde{x},~\E X\to \tilde{x}
$
exponentially fast. Moreover, it holds that 
\[
\begin{split}
L(\tilde{x}) &\le -\frac{1}{\beta}\log M_L(0)
-\frac{1}{2\beta}\log\left(1-\nu \right) \\
&  \le L_m+r(\beta)+\frac{\log 2}{\beta},
\end{split}
\]
where
\[
r(\beta):= -\frac{1}{\beta}\log M_L(0)-L_m \to 0,~\beta\to\infty.
\]
\end{theorem}

In the above result, we have used \eqref{eq:ass}  and the Laplace principle \eqref{eq:laplaceprinciple} so that $r(\beta)\to 0,~\beta\to\infty$.  How well $L(\tilde{x})$ approximates $L_m$ depends on how well $-\frac{1}{\beta}\log M_L(0)$ approximates $L_m$. When $\beta$ is sufficiently large and $L$ has a unique global minimizer, together with some reasonable assumptions put on $L$ and the probability density, one acually has
\cite{hsu1948theorem,McClure1983error,Inglot2014simple}
 \begin{gather}\label{eq:errorterm}
r(\beta) \le \frac{d}{2}\frac{\log(\beta/(2\pi))}{\beta}
+\frac{C_L}{\beta},~~\beta\gg 1,
 \end{gather}
 where $C_L$ does not depend on $d$. The convergence rate in Theorem \ref{thm:timecontinuous} does not depend on $d$ but in the error term \eqref{eq:errorterm}, there is linear $d$ dependence.
 For moderately high dimensional case, for example, in section \ref{sec: Rast} when $d=20$ and $\beta =O(10^2)$, one gets reasonably small error estimate.  
In theory, of course for any $d$, one can choose $\beta$ sufficiently large so that the error is still small. However, when $\beta$ is too large, $\exp(-\beta L)$ is near zero, and one loses lots of significant digits. On the other hand, when $\beta\gg d\gg 1$,  the average \eqref{eq:weightedmean} or \eqref{eq: ws exponential} becomes 
\[
\bar{x}^*=\argmin_{\Wj} L(\Wj),
\] which works similarly well as commented in section \ref{sec: MNIST},  though one cannot prove a theorem. In summary, our method can be feasible in high dimensional problems.

 Besides the largeness of $\beta$, how big $r(\beta)$ is also depends on the initial support of the law of $X_0$. If the probability that $X_0$ is near the minimizer is small, the  approximation quality is poor. This means for the algorithm to work well, the particles should explore the surrounding area well so that there is some probability that the neighborhood of the minimizer can be visited.

The assumptions \eqref{eq:ass} basically require $\lam, \beta$ to be large or $\sigma, V(0)$ to be small enough. Notice that the set of the parameters is not empty as one can control the initial variance $V(0)$. The assumption is restrictive in the above theorem, however this has to be understood as a proof of concept where other possible approaches may lead to improvements. In the first equation of \eqref{eq:ass}, $M_L(0)$ is also some quantity of the order $e^{-\beta L_m}$, so it essentially means $\lam\gtrsim \sigma^2$, which ensures that the variance of $X$ decays to zero. In fact, in the geometric Brownian motion, one needs $2\lam> \sigma^2$ to guarantee the variance of $X$ to vanish as shown in section \ref{sec: model algo}. Under such assumption, all particles will converge with exponential rate to a point which is within $O(1/\beta)$ of the global minimum. Moreover, recall $e^{-2\beta L_m} /M_L^2(0)\ge 1$. Hence, in the original CBO interacting particle system and their mean-field counterpart in \cite{carrillo2018analytical}, a large $d$ dependence in $\mu$ is sensitive for $\lambda$. This means that our model is more feasible and adapted for high dimensional optimization problems.


\subsection{The semi-discrete model}

Let us consider the CBO method \eqref{eq:optmodel} where $\bws$ is updated at only a number of
discrete time points. Alternatively, in Algorithm \ref{algo: main algo}, we let $\gamma$ fixed, take $N\to\infty$ and use the time continuous SDE to replace the discrete scheme at one iteration.

Define
\begin{gather}
t_k:=k\gamma.
\end{gather}
Then, one has the following semi-discrete model in the time-continuous setting, where a particle evolves according to the component geometric Brownian motion on interval $I_k:=[t_{k-1}, t_k)$, and the references $\bws$ is only updated on some discrete time points as
\begin{gather}\label{eq: semi-discrete}
dX=-\lambda(X-\bar{x}_k^*)\,dt+\sigma \sum_{i=1}^d (X-\bar{x}_k^*)_i dW_i \vec{e}_i \,, \quad t\in I_k\,,
\end{gather}
where  
\[
\bar{x}_k^*=\frac{\int_{\R^d} x\exp(-\beta L(x))\,\rho(x,t_{k-1})\,dx}{\int_{\R^d} \exp(-\beta L(x))\,\rho(x,t_{k-1})\,dx}.
\]
Similarly, $\rho(\cdot, t_k)$ means the law of $X$ at time $t_k$. 
We again consider the mean and variance of the model $m(t)=\mathbb{E}X$ and $V(t)=\E|X-\E X|^2$.
For this semi-discrete model, we have the following results, whose proof is given in Appendix \ref{app:semicont}:
\begin{proposition}\label{prop: semi discrete}
If the average sequence $\{\bar{x}_k^*\}$ is bounded and
\begin{gather}
\label{eq: ass on lam}
2\lambda>\sigma^2,
\end{gather} 
then the total weight defined in \eqref{eq:varianceandML} is bounded below 
\[
M_L^*:=\inf_{k}M_L(t_k)>0.
\]
 {Moreover, if step size $\gamma$ also satisfies that
\begin{gather}\label{eq:discretecondition}
e^{(-2\lambda+\sigma^2)\gamma}+\frac{e^{-\beta L_m}}{M_L^*}
(e^{(-2\lambda+\sigma^2)\gamma}-e^{-2\lambda\gamma})<1,
\end{gather}}
then  $m(t_k)\to \bar{m}$
and $V(t_k)\to 0$.
Consequently, $\bar{x}_k^*\to \bar{m}$ and  the law of $X$ converges weakly to $\delta(x-\bar{m})$ (i.e. in the dual of $C_b(\R^d)$, the space of bounded continuous functions equipped with the supremum norm).
\end{proposition}

\begin{remark}
\label{rmk:sigma}
Compared with Theorem \ref{thm:timecontinuous}, the choice of the parameters is much less restrictive. We only need (\ref{eq: ass on lam}) so that the variance can diminish to zero. However, we have assumed $\{\bar{x}_k^*\}$ to be bounded and  condition \eqref{eq:discretecondition}.
\end{remark}

\begin{remark}
\label{rmk:violation}
In actual numerical experiments, the condition on $\s, \lam$ is much loose than the 
theoretical condition. 
\end{remark}

To remove the assumption that $\bar{x}_k^*$ is bounded, one needs to estimate how $|\bar{x}_k^*|$ relies on the initial bounds of $|\bar{x}_0^*|$ so that the estimate can close up. For this discrete case, we have
\begin{multline*}
\frac{d}{dt}\mathbb{E}e^{-\beta L(X)}
=\beta\lambda \E(e^{-\beta L(X)}\nabla L(X)\cdot(X-\bar{x}_k^*))\\
+\frac{1}{2}\sigma^2\E\left[e^{-\beta L(X)}\left(\beta^2 \sum_i \partial_iL(X)^2(X-\bar{x}_k^*)_i^2
-\beta\sum_i (X-\bar{x}_k^*)_i^2\partial_{ii}L(X)\right)\right]
\end{multline*}
The issue is that the first term is hard to control now.  A possibility to overcome this difficulty is to study the fully discretized scheme, with the noise terms discretized using the Euler-Maruyama scheme. This will be explored elsewhere.


\section{Numerical Performance}
\label{numerics}
We assume $\lambda = 1$ in all the numerical examples in Sections \ref{sec: comp sgd}, \ref{sec: Rast} and \ref{sec: MNIST}. 

Let us first comment on some practical implementation aspects of the Algorithm \ref{algo: main algo}. The operator splitting to update all particles can be used in order to avoid overshooting. {One can choose to implement the algorithm as}
\[
\begin{split}
&\h{\Wj_{k}}=\bws_k+(\Wj_k-\bws_k)e^{-\lambda \g},\\
&\Wj_{k+1}= \h{\Wj_k} + \sigma \sqrt{\g} \sum_{i=1}^d\vec{e}_i\(\h{\Wj_k}-\bws_k\)_i z_i^j,
\end{split}
\]
where the first equation is the exact solution of the SDE
$
d\Wj = -\lam(\Wj-\ws)dt,
$
from $t = k\g$ to $t = k(\g+1)$.
 By overshooting, we mean $\h{\Wj_{k}}$ oscillates around $\bws_k$. This could bring instability as in the case of forward Euler in solving some stiff problems.

An alternative way to implement this step is to freeze $\bws_k$ in a time-step interval, then the geometric Brownian motion can be solved exactly by
\begin{gather}
\Wj_{k+1}=\ws_k+\sum_{i=1}^d\vec{e}_i(\Wj_k-\bws_k)_i\exp\( \(- \lambda - \frac{1}{2}\sigma^2\)\gamma+\sigma\sqrt{\gamma} \,z_i^j\)
\end{gather}
In practical simulations, this is comparable with the above splitting approach in most cases. 

Concerning the parameters in our CBO model \eqref{eq:optmodel}, one can observe that by increasing $\beta$ and decreasing $\sigma$ as iterations accumulate, the accuracy and convergence speed of the results  will be improved. The cooling strategy can be chosen to be similar to the annealing approach \cite{MR904050,MR995752,hwang1990large}. The intuition is that one decreases the temperature so that the system will cool down to the global minimum. Another practical strategy is to use larger $\sigma$ at early stages of the simulations for better exploration of the cost landscape, while use smaller $\sigma$ at later stages. For example, a possible strategy is to take
\[
\sigma_k=\sigma_0/\log(k+1)
\]
Decreasing $\sigma$ corresponds to decreasing the noise level. As it has been seen in the semi-discrete model \eqref{eq: semi-discrete}, we need $2\lambda>\sigma^2$ for the particles to concentrate. Hence, this strategy allows us to use large $\sigma>\sqrt{2\lambda}$ in the early stage to explore the surrounding area well.

We now show the performance of Algorithm \ref{algo: main algo} for our CBO model \eqref{eq:optmodel} in three model cases: an optimization of a test function with large oscillations and wide local minima in one dimension, a neural network for the MNIST data set and an optimization of a test function with many local minima in high dimension.

\subsection{Comparison with stochastic gradient descent (SGD) method}
\label{sec: comp sgd}
We first show an example where SGD can hardly find the global minimum, however, our method can easily find it. It has already been observed and proved that the geometry of the objective function will affect the performance of SGD method \cite{ZhuDaiSGD, Keskar2017, threefactors}. One of the reasons is that the expected exiting time for SGD to escape from a local minimum is exponentially proportional to the inverse of Hessian at the minimum, height of the basin and batch size. We construct the following optimization problem:
\begin{equation}
\label{eq: obj comp SGD}
\begin{aligned}
&\ell(x,\h{x}_i) = e^{\sin(2x^2)}+\frac{1}{10}(x - \h{x}_i - \frac{\pi}{2})^2, \qd \h{x}_i \sim \text{Normal}(0,0.1)\\
&L(x) = \frac{1}{n}\sum_{i = 1}^n\ell(x,\h{x}_i)
\end{aligned}
\end{equation}
The objective function $L(x)$ is plotted in Figure \ref{fig: obj42} for $n = 10^4$. It is easy to see that the global minimum is $\ws = \pi/2$. 
\begin{figure}[htbp]
\includegraphics[width=1\textwidth]{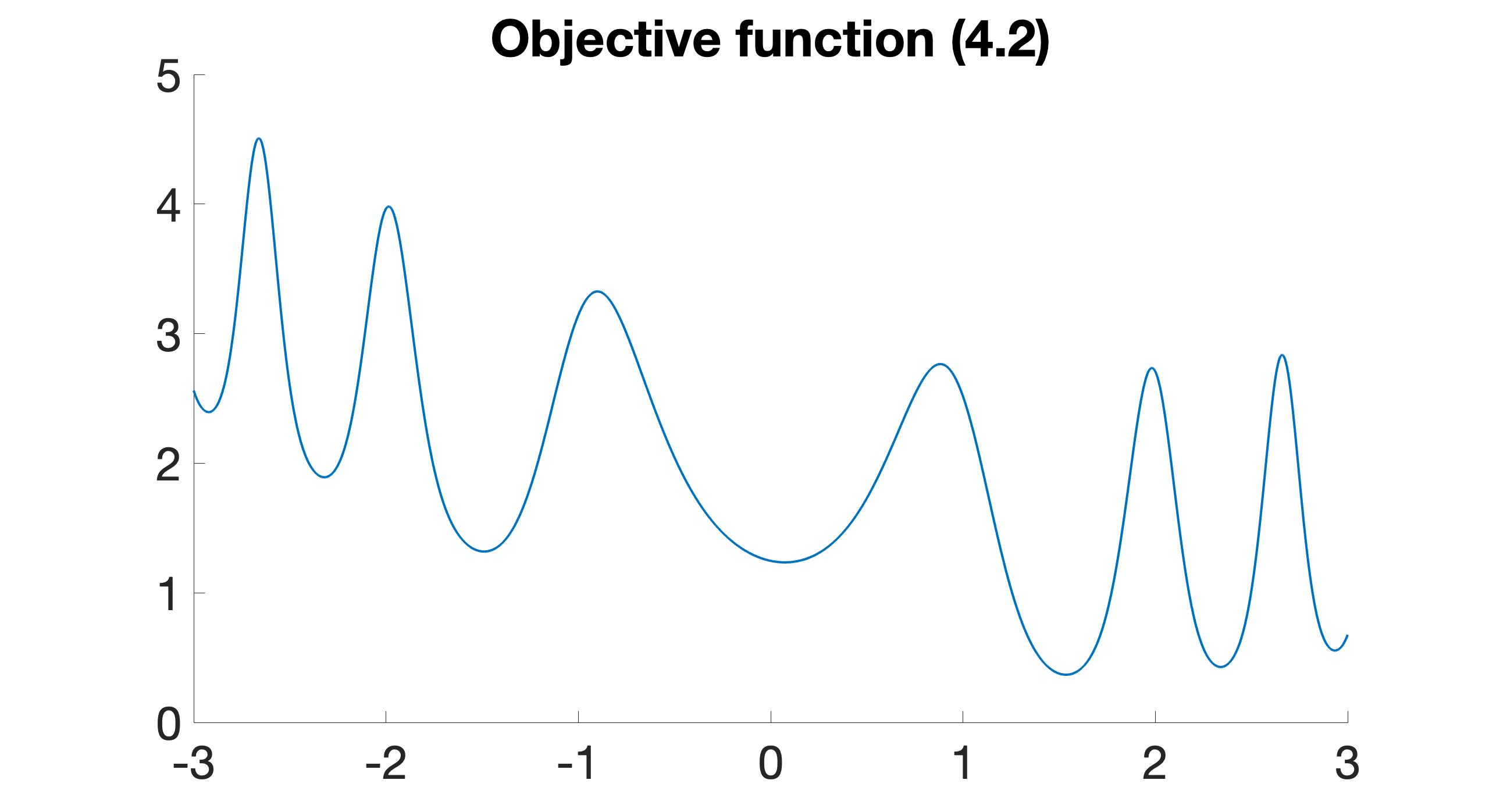}
\caption{}
\label{fig: obj42}
\end{figure}

SGD updates the parameter $x_k$ in the following way, 
\begin{equation*}
\begin{aligned}
	x_{k+1} = x_k - \frac{1}{m}\sum_{i\in b_k}\nb_x\ell(x_k,\h{x}_i),
\end{aligned}
\end{equation*}
where $b_k$ is an index set randomly drawn from $\{1, \cdots,n\}$. We can see there are many local minima with different shapes. Especially, the height of all the basins are large, some of the local minimum is much flatter than the global minimum, in which case SGD tends to be trapped in those local minima. However, the geometry of the objective function has little influence on our method. We show the success rate of both methods in Figure \ref{fig: suc obj42} with the same initialization and variables $n,m,\g$. We consider one simulation is successful if, $\lv \bws_k - \ws \rv < 0.25$ for our CBO or $\lv x_k - \ws \rv < 0.25$ for SGD, which means that our approximated minimizers is in one-half width of the global minimizer.
For both methods, we run 100 simulations and each simulation we run $10^4$ steps. In addition, we initialize $\W^j_0$ from uniform distribution in $[-3,3]$, and set 
\begin{equation*}
	\g = 0.01, \qd n = 10^4, \qd m = 20.
\end{equation*}
Besides, for our method, we set
\begin{equation*}
	N = 100, \qd M = 20, \qd \s = 5, \qd \b = 30,
\end{equation*}
and use the partial updates. For each simulation, the algorithm stops either when the stopping criterion in Step 3 of Algorithm \ref{algo: main algo} is satisfied with $\e = 10^{-3}$, or it finishes $10^4$ steps.

\begin{figure}[htbp]
\centering
\includegraphics[width=0.5\textwidth]{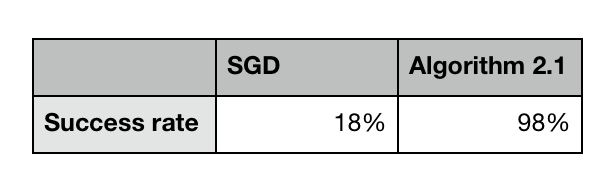}
\caption{The table shows the success rate of SGD and Algorithm \ref{algo: main algo}.}
\label{fig: suc obj42}
\end{figure}

From the table we can see that our method performs significantly better than SGD, the computational time for our method is a little longer than SGD though.  Notice here $M = 20$, that means,  there are only $20$ particles interacting with each other in each step, which is also computationally efficient. 


\subsection{The Rastrigin function in $d=20$}
\label{sec: Rast}
In this section, we compare our method with the one introduced in \cite{pinnau2017consensus}. 
The goal is to find the global minimum of the Rastrigin function, which reads
\begin{equation}
    \label{Rast fcn}
    L(x) = \frac{1}{d}\sum_{i = 1}^d\l[\l(x_i - B\r)^2 - 10\cos\l(2\pi(x_i-B)\r) + 10\r] + C,
\end{equation}
with $B = \text{argmin}\  L(x)$ and $C = \min L(x)$.

\begin{figure}[htbp]
\includegraphics[width=1\textwidth]{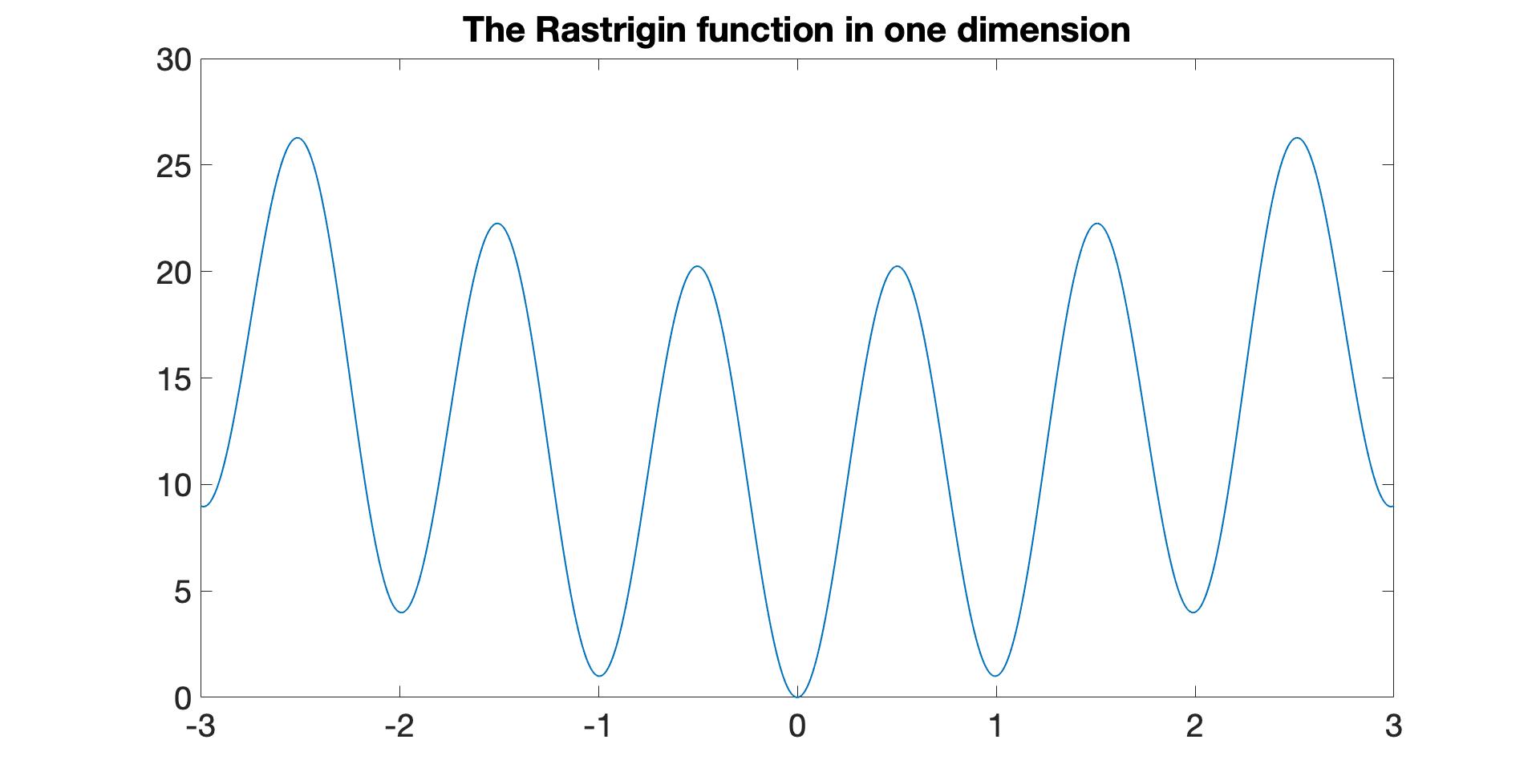}\\
\caption{}
\label{fig: Rast fcn}
\end{figure}

Figure \ref{fig: Rast fcn} shows the shape of the Rastrigin function $L(x)$ when $x\in \R^1, B = C = 0$ in \eqref{Rast fcn}. This illustration reveals that the local minima of this function are very close to the global minimum, so it is not easy for optimization algorithms to discern the location of the global minimum. The numerical experiments of \cite{pinnau2017consensus} indicate that their method performs well in finding the global minimum of the $one$-dimensional  and $twenty$-dimensional Ackley functions. However, compared to the Ackley function, the local minima of the Rastrigin function are much closer to the global minimum, so it is harder to find its global minimum. The performance of their method for $d=20$ is not good enough for the same set of parameters, as shown in \cite[Table 2]{pinnau2017consensus}. The success rate in their numerical experiments for $N=50, 100, 200$ is from $34\%$ to $64\%$.

In our numerical experiments, we set $C = 0$ and $B = 0, 1, 2$. We initialize all the particles uniformly on the interval $[-3,3]$. For the cases $B = 1,2$, the minimizer is not at the center of the initialization, which increases the difficulty of converging towards it. Besides we use partial updates in Step 3 and set
$\g = 0.01$ and $\beta = 30$.
For each simulation, we run $10^4$ steps. Our results are shown in Figure \ref{fig: R_fcn_20}. We display the success rate and averaged distance to the global minimum for 100 simulations. Notice that in  \cite[Table 2]{pinnau2017consensus}, $v_f(T)$ corresponds to $\bws_k$ in our paper if readers are interested in comparing performances. Here we consider one simulation is successful if the final $\bws_k$ is close to the global minimum $\ws$ in the sense that,
\begin{equation*}
	\lv (\bws_k)_i - (\ws)_i \rv < 0.25, \qd \text{for all } i,
\end{equation*}
which means our result is in one-half width of the global minimizer and $0.25$ is in order to keep consistency with \cite{pinnau2017consensus}. 
Figure \ref{fig: R_fcn_20} shows that our method is not only more efficient, but also performs better in terms of finding the global minimizer. Notice that although condition (\ref{eq: ass on lam}) is violated, as we mentioned in Remark \ref{rmk:violation}, the algorithm approaches the global minimum.  Also we notice that the success rate becomes worse when $N = 200$. This is possibly due to random fluctuation, but in general, larger $N$ gives better results but computationally more expensive. 
\begin{figure}[htbp]
\begin{center}
\includegraphics[width=0.9\textwidth]{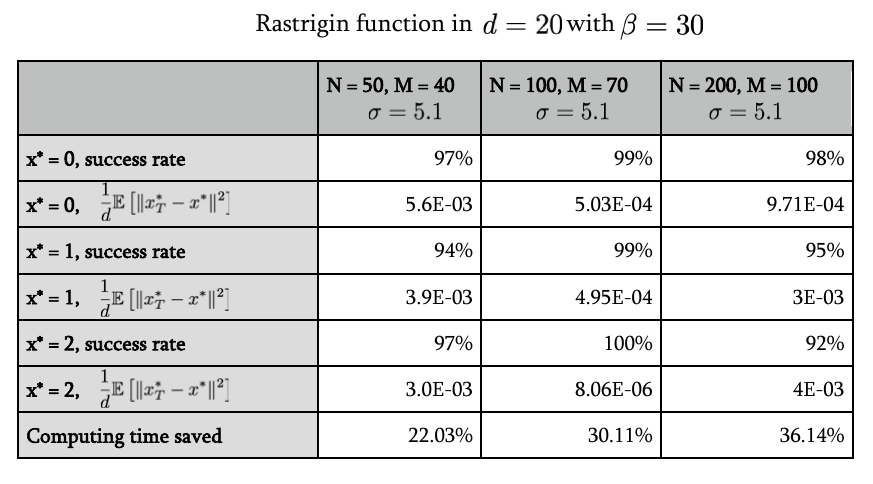}
\end{center}
\caption{This table shows the success rate and the error of our algorithm towards the global minimum for different Rastrigin functions with parameters leading to the global minimum being given by the constant vectors specified in each row. We also show the computational savings.}
\label{fig: R_fcn_20}
\end{figure}

\subsection{Experiments on MNIST data sets}\label{sec: MNIST}
In this section, we will run an optimization problem from machine learning, in order to show that our method also works for high dimensionality.
The MNIST data is a set of pictures with grayscale numbers from $0$ to $9$. The input data $\h{x}$ is a vector of dimension 728, which records the Grayscale of each pixel.  The output data $\h{y}\in \{\vec{e}_k\}_{k=1}^{10}$, where $\vec{e}_k$ is a vector of dimension $10$ with only the $k$-th element $1$ and $\vec{e}_k$ represents that it is a picture of number $k-1$. We use the Neural Network without hidden layers to model this classification problem, the function defining the neural network is given by
\begin{equation*}
    f(x, \h{x}) = a( ReLu(\theta \h{x} + B)), \qd x = (\theta, B),
\end{equation*}
is a function depending on the parameter $x$ and mapping $\h{x}\in \R^{728}$ to $\R^{10}$. Here $\theta \in \R^{10\times 728}, B\in \R^{10}$. $ReLu(x) = x\mathbb{1}_{x\geq0}=((x_i)_+)_{i\in\{1,\dots,728\}}$ is an activation function with $(r)_+$ being the positive part of the number $r$, while $a(x)$ is another activation function called {\it softmax}, which reads
\begin{equation*}
    a({\bf x}) = \frac{e^{x_j}}{\sum_j e^{x_j}}.
\end{equation*}
So that the $j-$th component of $f$ represents the probability of $\h{x}$ being the image associated to number $j-1$. The objective function to be minimized is the cross entropy loss, 
\begin{equation}
\label{eq: training error}
    L(x) = \frac{1}{n}\sum_{i=1}^n\ell(f(x,\h{x}_i),\h{y}_i), \qd \ell(f,y) =-\sum_{k=1}^{10} \h{y}_k \log(f_k),
\end{equation}
where the observations belong to the subset $\h{y}\in \{\vec{e}_k\}_{k=1}^{10}$. 

In the setup of deep learning, one uses deep neural network to construct the model. As the neural network gets deeper or wider, the dimension $d$ of parameter $\w$ becomes very large, $d \gg n \gg 1$, which results in potentially many local minima.  So the goal here is not only finding the global minimum, but also the good global minimum. It is common practice to quantify the quality of the minimum by the test accuracy. We take the largest component as the prediction of our model at $\h{x}$, that is,
\begin{equation*}
\label{eq: prediction}
\begin{aligned}
	g(\ws,\h{x}) = \vec{e}_j, \qd \text{where }j = \argmax_i{f(\ws,\h{x})_i }
\end{aligned}
\end{equation*}
and $f(\ws,\h{x})_i$ means the $i$-th component. We further define the test accuracy by 
\begin{equation}
\label{eq: test error}
\begin{aligned}
       \text{accuracy}_{test}(\ws) = \frac{1}{p}\sum_{i= 1}^{p} \mathbb{1}_{\{g(\ws,\h{x}^\tt_i)=\h{y}^\tt_i\}},
\end{aligned}
\end{equation}
where $p$ is the size of the test set,  number of data in the test set.


Let us first discuss how different elements in the algorithm influence the convergence rate. For all experiments, we use exactly the same initialization, which is drawn from standard normal distribution. Besides, we use the full updates in Step 3 and set  the following values as reference case for our simulations,
\begin{equation}\label{reference}
\begin{aligned}
    &N = 100, \qd M = 10, \qd n = 10^4, \qd m = 50, \qd\g = 0.1, \qd \s = \sqrt{0.1}, \qd \lam = 1.
\end{aligned}
\end{equation} 
Here $N$ is the number of total particles, $M$ is the batch size used to update $\bws$; $n$ is the number of total training data, $m$ is the batch size used to calculate the estimated objective function $\hL_j$;  $\g, \s,\lam$ are the learning rate, the noise rate and drift rate respectively. As mentioned in Remark \ref{rmk: step3}, we allow all the particles to do an independent Brownian motion with variance $\s$ when $\bws$ stops updating and then the algorithm repeats until stabilization. 

In both Figures \ref{fig: summarize_para1} - \ref{fig: summarize_para3}, the x-axis represents for the number of epoch. Here one epoch is equal to $n/m$ steps, which means we go through the whole training data set: $200$ steps for $m = 50$ and $1$ step for $m = n$.  The y-axis represents the test accuracy as defined in (\ref{eq: test error}) over $p=10^4$ data sets. The test data set elements $(x_i^{test}, y_i^{test})_{i=1}^{10000}$ are all different from the training data set elements we used in the objective function \eqref{eq: training error}. 

\begin{figure}[htbp]
\includegraphics[width=1\textwidth]{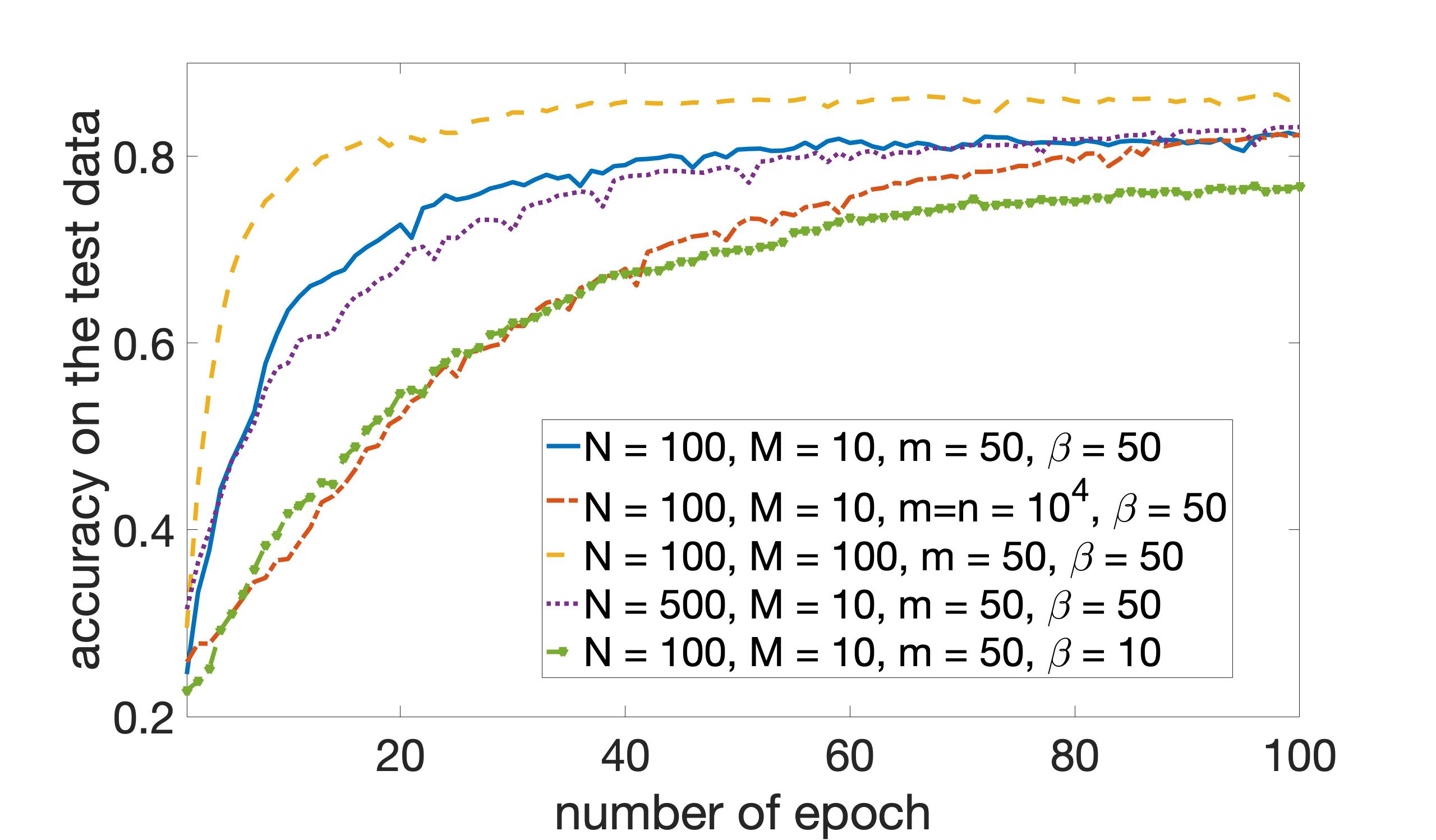}\\
\caption{Comparison of the performance between the reference solution with parameter in \eqref{reference} and first set of data with different parameters as explained in the inlet.}
\label{fig: summarize_para1}
\end{figure}

In Figure \ref{fig: summarize_para1} we compare the performance over the neural network of the reference solution with the parameters in \eqref{reference} with respect to other set of parameters. The main general observations inferred are:
\begin{itemize}
\item Using estimated value $\hLj$ is better than using the exact value $L_j$. The small batch $m = 50$ not only saves $99.5\%$ of the computational cost when calculating $L^j$, the increase rate of the accuracy is even better than the full batch. 

\item Using more particles in the interaction, larger $M$, to compute the average $\bws$ increases the overall performance as expected. However, we have to point out that the case where $M = 10$ save $40\%$ of computing time compared to the case $M = 100$. 

\item Our method with $100$ particles and $10$ interacting particles at each step already gives good accuracy. We will show later in Figure \ref{fig: summarize_para3} that as the number of particles goes to $10^3$, the accuracy will be comparable to SGD. However, compared to the performance of $N= 100$ and $N=500$, the accuracy improved slowly as the number of particles becomes larger.

\item Larger $\a$ corresponds to faster convergence rate. As $\a$ becomes larger, it achieves better accuracy faster. However, $\a$ cannot be too large, otherwise $\mu^j$, defined as in (\ref{eq: ws exponential}), is smaller than the minimal threshold positive value for the computer, which makes $\bws$ infinity. 

\end{itemize}

In Figure \ref{fig: summarize_para2} we compare our results with parameters \eqref{reference} with the simulations using the variations of the new CBO model \eqref{eq:optmodel} discussed in Section 2. We use the variant of the CBO model without the noise term and the variant of the CBO model using the minimal value of the cost \eqref{eq: ws argmin} over the agents rather than the average $\bws$. We deduce the following general observations:
\begin{itemize}

\item A similar behavior in the performance with or without the last term in (\ref{discrete swarming}) is obtained. Since there is already stochasticity involved when calculating $\hLj$ and selecting random subsets of particles, it is usually not necessary to add extra noise when updating the positions of particles for these machine learning problems. However, for other optimization problems like the numerical experiment in Section \ref{sec: Rast}, the geometric Brownian noise seems necessary to avoid clogging in local minima.

\item  Using $\argmin \hL(\Wj)$ to update $\ws_k$ has also a similar performance.
\end{itemize}

In Figure \ref{fig: summarize_para3}, we compare our results with stochastic gradient descent method. We use the same data set and neural network structure. We set the learning rate equal to $0.1$, which is the same as $\g$ in our new method. To make a fair comparison, we run $1000$ simulations of SGD with different initializations follow from standard normal distributions, which is also the same initialization as the proposed method. We plot the best one among all SGD simulations. We can see that our method with $1000$ particles, which have the same computational cost, is slightly better than the best SGD over $1000$ simulations. Besides, if we use $N=10^4$, the test accuracy could be improved to around 90\%. Therefore, our method can potentially get comparable accuracy with SGD in some settings.

As a concluding remark, we have shown in our numerical examples that, two alternative numerical methods, one where only random batch is involved without the diffusion term, one where $\bar{x}^*_k$ is directly equal to the argmin of the particles' value, performs as well as our method. However, the theoretical proof will be left for future study.

\begin{figure}[htbp]
\includegraphics[width=1\textwidth]{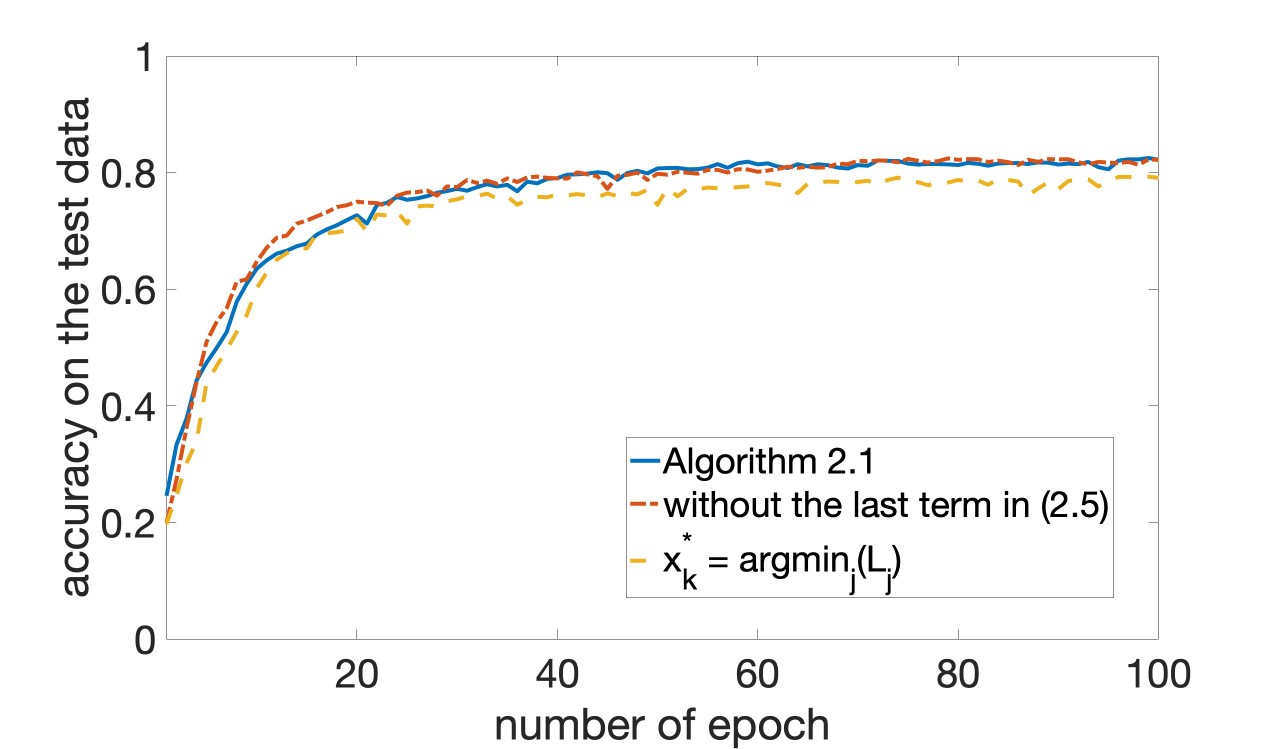}\\
\caption{Comparison of the performance between our new CBO algorithm and its variants.}
\label{fig: summarize_para2}
\end{figure}

\begin{figure}[htbp]
\includegraphics[width=1\textwidth]{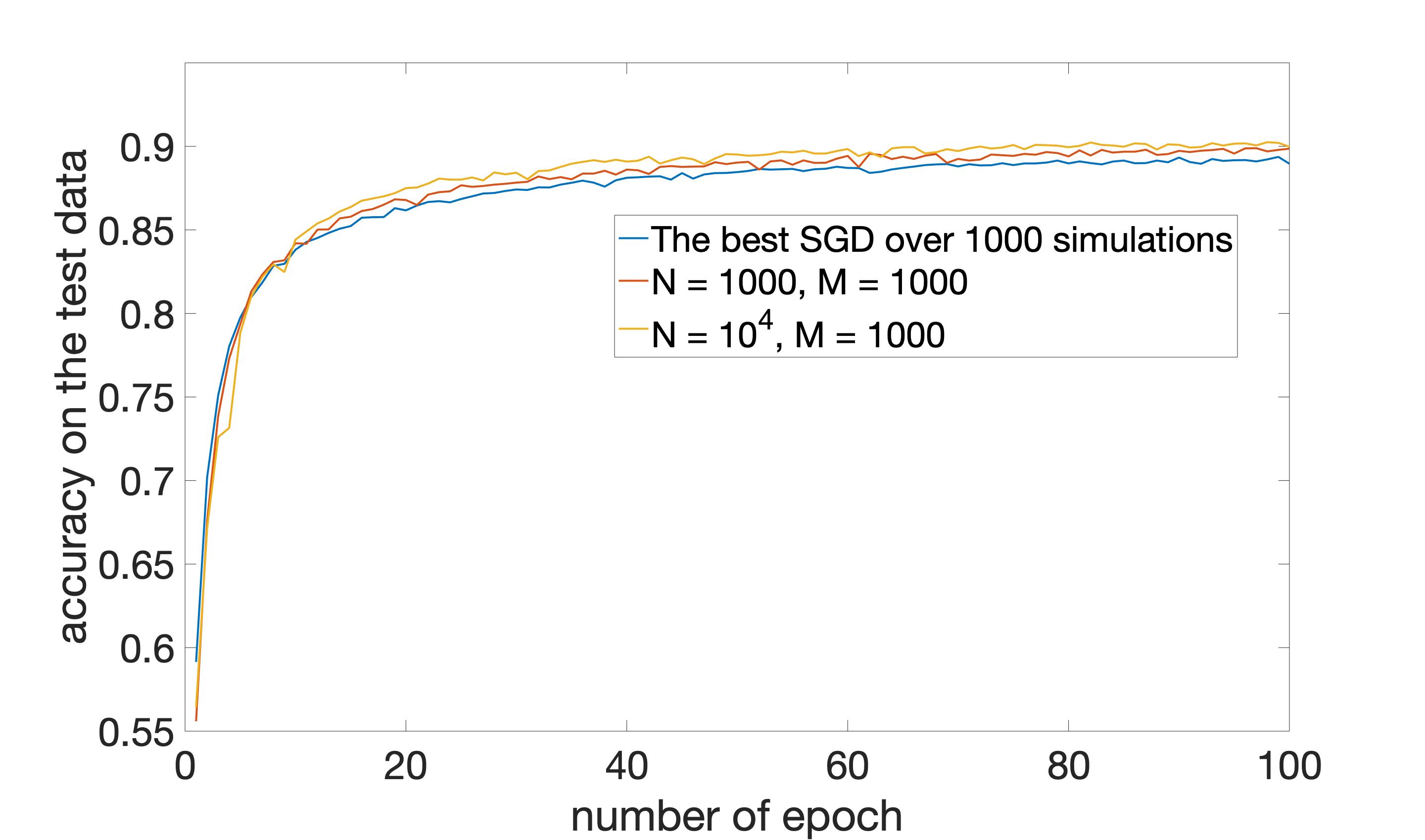}\\
\caption{Comparison of our new CBO algorithm and SGD.}
\label{fig: summarize_para3}
\end{figure}


\section{Conclusion}
We improve the gradient-free optimization method upon \cite{pinnau2017consensus}, to make it effective for high dimensional optimization problems. We show in Theorem \ref{thm:timecontinuous} and Proposition \ref{prop: semi discrete} that because of the component-wise geometric Brownian motion, the mean field limit of the method always converges to its good approximation of the global minimum with all the parameters independent of dimensionality. We show in Section \ref{sec: MNIST} that for the MNIST data with two layers Neural Network, which is a $7290$-dimensional optimization problem, only with $100$ particles, it can already achieve 82\% accuracy. In another example \ref{sec: Rast}, our method has a significant higher success rate in finding the global minimum of the Rastrigin function compared to the practical implementation of the original method introduced in \cite{pinnau2017consensus}. 

There are still lots of open problems in this direction, among them:
\begin{itemize}
\item Theoretical study for the method. Our theorems have not involved the random batch on particles and random batch on the data set. How the random batch method will affect the process of finding the global minimum will be the object of future studies.
\item Stability condition for the method and criteria of choosing the optimal variables, such as, $M,N,\a, \sigma, \gamma$ remain to be understood. 
\end{itemize}

\section*{Acknowledgement}
JAC was partially supported by the EPSRC grant number EP/P031587/1.
SJ was partially supported by NSFC grants No. 11871297 and No. 31571071. LL was partially supported by NSFC grants No. 11901389 and No. 11971314, and the Shanghai Sailing Program 19YF1421300. 
We thank Dr. Doheon Kim for discussion on the convergence rate of the Laplace principle.

\appendix


\section{Proof of Theorem \ref{thm:timecontinuous}}\label{app:meanfield1}

To prove this theorem, we need some preparation. 
\begin{lemma}
$V(t)$ and $M(t)$ satisfy the following:
\begin{subequations}
\begin{align}
& \frac{d}{dt}V(t) \le -\left(2\lambda-\sigma^2-\sigma^2\frac{e^{-\beta L_m}}{M_L(t)}\right)V(t), \label{eq:Vineq}\\
& \frac{d}{dt}M_L^2(t)\ge -2\beta c_L(2\lambda+\sigma^2)e^{-2\beta L_m} V(t).
\end{align}
\end{subequations}
\end{lemma}

\begin{proof}
By It\^o's calculus, it holds that
$$
d|X-\E X|^2=2(X-\E X)\cdot dX-2(X-\E X)\cdot d\E X
+\sigma^2 \sum_{i} (X-\bws)_i^2 dt.
$$
Hence, one can deduce that
\[
\begin{split}
\frac{d}{dt}V(t) &=-2\lambda \E \Big[ (X-\E X)\cdot(X-\bws)\Big]
+\sigma^2\E|X-\bws|^2\\
&=-2\lambda V(t)+\sigma^2\E|X-\bws|^2.
\end{split}
\]
Here, we used the fact $\E(X-\E X)\cdot(\E X-\bws)=0$. Moreover, we get
$\E|X-\bws|^2=V(t)+|\E X-\bws|^2$. By Jensen's inequality, for any $a\in \R^d$,
\begin{gather}\label{eq:jensen}
|a-\bws|^2 \le \frac{\E_{X_1\sim X} |a-X_1|^2 e^{-\beta L(X_1)}}{M_L},
\end{gather}
where $X_1\sim X$ means $X_1$ has the same distribution as $X$. Therefore,
\begin{gather}\label{eq:apppfaux2}
\E|X-\bws|^2\le  V(t)+e^{-\beta L_m}\frac{V(t)}{M_L(t)},
\end{gather}
and \eqref{eq:Vineq} follows.  {We remark that if $L$ a constant, the right hand side of \eqref{eq:apppfaux2} is $2V(t)$. However, in this case,
$\bws=\E X$ and the upper bound should be $V(t)$ instead of $2V(t)$. The reason is that Jensen's inequality \eqref{eq:jensen} is lost for this case.}

Analogously, by It\^o's calculus, one can infer that
\[
\begin{split}
&d\E e^{-\beta L(X)}=-\beta \E e^{-\beta L(X)} \nabla L(X)\cdot dX \\
& +\frac{1}{2}\sigma^2 \sum_i \E \Big[ e^{-\beta L(X)}(X-\bws)_i^2 (\beta^2 (\partial_i L)^2
-\beta \partial_{ii}L(X)) \Big]\,dt =: (I_1+I_2)dt.
\end{split}
\]
By definition and Assumption \ref{ass:modelcond}, one obtains
\[
I_1=\beta\lambda \E e^{-\beta L(X)}(\nabla L(X)-\nabla L(\bws))\cdot (X-\bws)
\ge -\beta \lambda e^{-\beta L_m} c_L \E |X-\bws|^2
\]
where $\E e^{-\beta L(X)}\nabla L(\bws)\cdot (X-\bws)=0$ is used.
For $I_2$, one has
\[
I_2\ge \frac{1}{2}\sigma^2(-\beta c_L)e^{-\beta L_m} \E|X-\bws|^2.
\]
Hence, we conclude that
\begin{equation}\label{xxx}
\frac{dM_L}{dt}\ge -\beta c_L e^{-\beta L_m}(\lambda+\frac{1}{2}\sigma^2)
\E |X-\bws|^2.
\end{equation}
Using the same estimate as in \eqref{eq:jensen}, one finds
\begin{gather}\label{eq:aux1}
\E |X-\bws|^2 \le V(t)+\frac{e^{-\beta L_m}}{M_L(t)}V(t)\le 2\frac{e^{-\beta L_m}}{M_L(t)}V(t).
\end{gather}
Inserting \eqref{eq:aux1} into the differential inequality \eqref{xxx}, the desired estimate follows.
\end{proof}

\begin{proof}[Proof of Theorem \ref{thm:timecontinuous}]
Define
\[
T:=\sup\left\{t: M_L(s)\ge \frac{1}{2}M_L(0),~\text{for all } s\in [0, t] \right\}.
\]
Clearly, $T>0$. Assume that $T<\infty$. Then, for $t\in [0, T]$, by the assumption on $\mu$ in \eqref{eq:ass}, one can deduce that
\[
2\lambda-\sigma^2-\sigma^2\frac{e^{-\beta L_m}}{M_L(t)}
\ge 2\lambda-\sigma^2-2\sigma^2\frac{e^{-\beta L_m}}{M_L(0)}=\mu>0.
\]
Consequently, one has
\[
\frac{dV}{dt}\le -\mu V(t).
\]
and thus
\[
V(t)\le V(0)\exp(-\mu t).
\]
Hence, by the assumption on $\nu$ in \eqref{eq:ass},
\begin{gather*}
\begin{split}
M_L^2(t) &\ge M_L^2(0)-2\beta c_L(2\lambda+\sigma^2)e^{-2\beta L_m}
V(0)\int_0^t e^{-\mu s}\,ds\\
&> M_L^2(0)-\frac{2V(0)\beta c_L(2\lambda+\sigma^2)e^{-2\beta L_m}}{\mu}\ge \frac{1}{4}M_L^2(0).
\end{split}
\end{gather*}
This means that there exists $\delta>0$ such that
$M_L^2(t)\ge \frac{1}{4}M_L^2(0)$ in $[T, T+\delta)$ as well.
This then contradicts with the definition of $T$. Hence,  $T=\infty$.
Consequently, 
\begin{gather}\label{eq:Vt}
V(t)\le V(0)\exp(-\mu t)
\end{gather}
holds and 
\begin{gather}\label{eq:Mt}
M_L(t)>\frac{1}{2}M_L(0)
\end{gather}
for all $t>0$.
Using again Jensen's inequality \eqref{eq:jensen} and \eqref{eq:Vt}-\eqref{eq:Mt}, we infer that
\[
\left|\E X-\bws\right|^2\le \frac{\E_{X_1\sim X}|X_1-\bws|^2 e^{-\beta L(X_1)}}{M_L(t)}
\le e^{-\beta L_m}\frac{V(t)}{M_L(t)}\le C\exp(-\mu t).
\]
Moreover, one has
\[
|\frac{d}{dt}\E X|\le \lambda \E |X-\bws|
\le \lambda \sqrt{\E |X-\bws|^2}\le \lambda \sqrt{V(t)+|\E X-\bws|^2}
\le C\exp(-\mu t/2).
\]
Since the right-hand side is integrable on time, it follows that $\E X\to \tilde{x}$ for  some $\tilde{x}$
and $\bws\to \tilde{x}$, with exponential rate.
Since $\E X\to \tilde{x}$ and $V(t)\to 0$, $M_L(t)\to e^{-\beta L(\tilde{x})}$. Hence, we deduce that
\[
e^{-2\beta L(\tilde{x})} > M_L^2(0)(1-\nu).
\]
Therefore, we conclude that
\[
L(\tilde{x})<-\frac{1}{\beta}\log M_L(0)-\frac{1}{2\beta}\log(1-\nu).
\]
By the assumption on $\nu$ in \eqref{eq:ass}, one thus has
\[
L(\tilde{x})<-\frac{1}{\beta}\log M_L(0)+\frac{\log 2}{\beta}.
\]
By the Laplace principle \eqref{eq:laplaceprinciple}, $r(\beta)=-\frac{1}{\beta}\log M_L(0)-L_m\to 0$. See more details in \cite{carrillo2018analytical}.
\end{proof}

\section{Proof of Proposition \ref{prop: semi discrete}}\label{app:semicont}
\begin{proof}
Since $\bar{x}_k^*$ is constant during the time interval, we find
that the mean value $m(t)$ satisfies
\[
\frac{d}{dt}(m(t)-\bar{x}_k^*)=-\lambda(m(t)-\bar{x}_k^*)
\]
Therefore, we get
$
m(t_k)-\bar{x}_k^*= (m(t_{k-1})-\bar{x}_k^*)e^{-\lambda \gamma}.
$
Hence, one obtains
$
m(t_k)=m(t_{k-1})e^{-\lambda \gamma}+\bar{x}_k^*(1-e^{-\lambda\gamma}).
$
Consequently, it holds that
\[
m(t_k)=m_0e^{-\lambda k\gamma}+\sum_{\ell=0}^{k-1} (1-e^{-\lambda\gamma})\bar{x}_\ell^*e^{-(k-\ell)\lambda\gamma} .
\]
If $\bar{x}_{\ell}^*$ is bounded, this summation converges, and the sum is controlled by
$C\sup_k\|\bar{x}_k^*\|$ with $C$ independent of $\gamma$.
By It\^o's calculus, the second moment satisfies
\[
\frac{d}{dt}\mathbb{E}|X-\bar{x}_k^*|^2=(-2\lambda+\sigma^2)\mathbb{E}|X-\bar{x}_k^*|^2.
\]
Since the variance is given by
\[
V(t_k)=\mathbb{E}|X-\bar{x}_k^*|^2-|m(t_k)-\bar{x}_k^*|^2, 
\]
we have
\[
V(t_k)=V(t_{k-1})e^{(-2\lambda+\sigma^2)\gamma}+(e^{(-2\lambda+\sigma^2)\gamma}
-e^{-2\lambda\gamma})|m(t_{k-1})-\bar{x}_k^*|^2.
\]
If  $\{\bar{x}_k^*\}$ is bounded, then $m(t_k)$ is bounded as has been shown.  {Consequently, $V(t_k)$ is bounded uniformly. It then follows that there exists a compact set $K$ such that
\[
\sup_k \int_{\R^d\setminus K} d\rho(x, t_k)<1/2.
\]
Hence, we obtain that
\[
M_L^*\ge \frac{1}{2}\inf_{x\in K}e^{-\beta L(x)}>0.
\]}
 {Using \eqref{eq:jensen}, we have
\[
|m(t_{k-1})-\bar{x}_k^*|^2 \le \frac{e^{-\beta L_m}}{M_L}V(t_{k-1})
\le \frac{e^{-\beta L_m}}{M_L^*}V(t_{k-1}).
\]}
If $\gamma$ satisfies condition \eqref{eq:discretecondition}, then one sees easily that $V(t_k)\to 0$
and $|m(t_{k-1})-\ws_k|\to 0$. Hence, it is clear that
\[
\bar{m}:=\lim_{k\to\infty}m(t_k)
\]
exists.

Using Chebyshev's inequality, it is easy to see that for any $\epsilon>0$, there exists $R>0$ such that
\[
\sup_{t\ge 0}\E 1_{X\in \R^d\setminus B(0,R)}\le \epsilon.
\]

For any test function $\varphi\in C_b$, we find $R>|\bar{m}|$, $\phi\in C_b^2(\R^d)$ such that $\|\phi\|_{C_b}\le 2\|\varphi\|_{C_b}$, and that
\[
\sup_{x\in B(0, R)} |\phi-\varphi |\le \epsilon,
\]
and
\[
\sup_{t\ge 0}\E 1_{X\in \R^d\setminus B(0,R)}\le \frac{\epsilon}{\|\varphi\|_{C_b}}.
\]
Then, we deduce that
\[
|\E \varphi(X)-\varphi(\bar{m})|\le |\E \phi(X)-\varphi(\bar{m})|+\E |\phi(X)-\varphi(X)|\to 
|\phi(\bar{m})-\varphi(\bar{m})|+\E |\phi(X)-\varphi(X)|\le 2\epsilon.
\]
Consequently, we have
\[
\E e^{-\beta L(X)}\to \E e^{-\beta L(\bar{m})}>0.
\]
Hence, we conclude that
\[
\inf_{t>0}\E e^{-\beta L(X)}>0.
\]
Finally, using similar estimates as in the time continuous case, we obtain that
\[
|\bar{x}_k^*-m(t_k)|^2\to 0,
\]
consistent with the fact that $\bar{x}_k^*$ is bounded.
This finishes the proof.
\end{proof}

\bibliographystyle{plain}

\end{document}